 \documentclass{amsart}
\usepackage{amsmath,amssymb}

\newtheorem{theorem}{Theorem}[section]

\numberwithin{equation}{section}

\def\DJ{{\hbox{{\thinspace}D\kern-.8em\raise.15ex\hbox{--}\kern.35em}}}
\def\DJo{{$\;$\kern-.4em{\DJ}okovi\'c}}

\def\NSERC{Supported in part by an NSERC Discovery Grant.}

\renewcommand{\subjclassname}{\textup{2000} Mathematics Subject
Classification}

\begin{document}
\addtocounter{section}{1}

\title[Generalization of Mirsky's theorem]
{Generalization of Mirsky's theorem on diagonals and eigenvalues of matrices}

\author {Dragomir \v{Z}. \DJo}
\address{Department of Pure Mathematics and Institute for Quantum Computing, University of Waterloo,
Waterloo, Ontario, N2L 3G1, Canada}
\email{djokovic@uwaterloo.ca}
\thanks{\NSERC}


\date{}

\begin{abstract} Mirsky proved that, for the existence of a 
complex matrix with given eigenvalues and diagonal entries, the 
obvious necessary condition is also sufficient. We generalize 
this theorem to matrices over any field and provide a short 
proof. Moreover, we show that there is a unique 
companion-matrix-type solution for this problem.
\end{abstract}

\maketitle 
\subjclassname{ 15A18, 15A29 }
\vskip5mm

If $\lambda_1,\ldots,\lambda_n$ are the eigenvalues of a 
complex matrix $A$ of order $n$ and $d_1,\ldots,d_n$ are its 
diagonal elements, then the sum of the $\lambda_i$'s is 
necessarily equal to the sum of the $d_i$'s. 
Mirsky \cite{LM} proved that the converse holds. 
If the data are real numbers, he proved that $A$ can be chosen 
to be real as well. 
For a recent short proof of Mirsky's results see \cite{CL}. 

Instead of specifying the eigenvalues of a matrix we shall 
specify its characteristic polynomial. We shall work over any 
field $F$. Let 
\begin{equation*} \label{kar-pol}
f(t)=t^n+c_{n-1}t^{n-1}+c_{n-2}t^{n-2}+\cdots+c_0
\end{equation*}
be a monic polynomial over $F$.

\begin{theorem} \label{PrvaTeorema}
Given a sequence $d_1,\ldots,d_n$ in $F$ with 
$d_1+\cdots+d_n=-c_{n-1}$, there exists a unique sequence 
$b_1,\ldots,b_{n-1}$ in $F$ such that $f(t)$ is the characteristic polynomial of the matrix
\begin{equation*} \label{moja-mat}
A=[a_{ij}]=\left[ \begin{array}{cccccc}
d_1 & 0   & 0   & \cdots & 0       & b_1 \\
1   & d_2 & 0   &        & 0       & b_2   \\
0   & 1   & d_3 &        & 0       & b_3   \\
\vdots & & & & & \\
0   & 0   & 0   &        & d_{n-1} & b_{n-1} \\
0   & 0   & 0   &        & 1       & d_n \end{array} \right].
\end{equation*}
\end{theorem}

\begin{proof}
For any $S\subseteq\{1,\ldots,n\}$, let $A_S$ be
the submatrix of $A$ whose entries are the $a_{ij}$ 
with $i,j\in S$, and let $M_S=\det A_S$. 
It suffices to prove that the system of $n-1$ equations 
\begin{equation*} \label{sistem}
\sum_{S:|S|=n-k+1} M_S=(-1)^{n-k+1} c_{k-1}, \quad k=1,\ldots,n-1
\end{equation*}
in $n-1$ unknowns $b_1,\ldots,b_{n-1}$ has a unique solution 
in $F$. 
(By $|S|$ we denote the cardinality of $S$.) 
This is true because of the following claim: 
for each $k$ we have
\begin{equation*} \label{oblik-jed}
\sum_{S:|S|=n-k+1} M_S=(-1)^{n-k} b_k+g_k, 
\end{equation*}
where $g_k$ is a polynomial in the unknowns 
$b_{k+1},\ldots,b_{n-1}$ only. 

To prove this claim, it suffices to show that if $b_k$ occurs in 
$M_S$ and $|S|\le n-k+1$ then $S=\{k,\ldots,n\}$. 
By the hypothesis $b_k$ occurs in $M_S$, and so  
$\{k,n\}\subseteq S$ and there must exist a permutation $\pi$ 
of $S$ such that $\pi n=k$ and 
\begin{equation*} \label{proizvod}
\prod_{i\in S} a_{\pi i,i}\ne0. 
\end{equation*}
Hence, $\pi i\in\{i,i+1\}$ for $i\in S\setminus\{n\}$. 
As $k\in S$ and $\pi n=k$, we have $\pi k=k+1$. 
If $k<n-1$ then $\pi(k+1)\ne k+1$ and so $\pi(k+1)=k+2$, etc. 
By repeating this argument, we conclude that 
$\{k,\ldots,n\}\subseteq S$. As $|S|\le n-k+1$, it follows that 
$S=\{k,\ldots,n\}$.
This completes the proof of our claim and the theorem.
\end{proof}

If $d_n=-c_{n-1}$ and all other $d_i=0$, then $A$ becomes the 
well known Frobenius companion matrix of $f(t)$:
\begin{equation*} \label{Frob-mat}
C=\left[ \begin{array}{cccccc}
0   & 0   & 0   & \cdots & 0  & -c_0 \\
1   & 0   & 0   &        & 0  & -c_1   \\
0   & 1   & 0   &        & 0  & -c_2   \\
\vdots & & & & & \\
0   & 0   & 0   &        & 0  & -c_{n-2} \\
0   & 0   & 0   &        & 1  & -c_{n-1} \end{array} \right].
\end{equation*}

The next theorem, which has much simpler proof, provides explicit 
formulae for the unknown elements $b_1,\ldots,b_{n-1}$ and 
establishes the existence assertion of Theorem \ref{PrvaTeorema}, but not the uniqueness.

For $k=0,1,\ldots,n-1$, denote by $h_r(d_1,\ldots,d_k)$ the sum 
of all monomials in $d_1,\ldots,d_k$ of degree $r$. 
(In particular $h_0(d_1,\ldots,d_k)=1$). Let $I_n$ be the 
identity matrix of order $n$.

\begin{theorem} \label{DrugaTeorema}
If $d_n=-c_{n-1}-d_1-\cdots-d_{n-1}$ and
\begin{equation} \label{formule-b}
b_k=-\sum_{i=k-1}^n c_i h_{i-k+1}(d_1,\ldots,d_k), \quad
k=1,\ldots,n-1,
\end{equation}
where $c_n=1$, then $\det(tI_n-A)=f(t)$. 
\end{theorem}
\begin{proof}
Let $T=[t_{ij}]$ be the upper triangular matrix with entries 
$t_{ij}=h_{j-i}(d_1,\ldots,d_i)$, $1\le i\le j\le n$. 
As all $t_{ii}=1$, $T$ is invertible. 
It suffices to verify that $AT=TC$, which is straightforward. 
\end{proof}

For example, if $n=4$ then the above formulae read
\begin{eqnarray*}
b_1 &=& -c_0-c_1d_1-c_2d_1^2-c_3d_1^3-d_1^4,\\
b_2 &=& -c_1-c_2(d_1+d_2)-c_3(d_1^2+d_1d_2+d_2^2)
-(d_1^3+d_1^2d_2+d_1d_2^2+d_2^3),\\
b_3 &=& -c_2-c_3(d_1+d_2+d_3)
-(d_1^2+d_2^2+d_3^2+d_1d_2+d_1d_3+d_2d_3).
\end{eqnarray*}

\end{document}